\documentclass{amsart}

\usepackage{geometry}  
\usepackage{amsmath,amsthm,amssymb}
\usepackage{enumitem,thmtools}
\usepackage{xfrac}
\usepackage{ytableau}
\usepackage{tikz, calc, forest}

\usetikzlibrary{arrows.meta,shapes}
 
\newcommand{\lp}{\left(}
\newcommand{\rp}{\right)}

\newcommand{\tx}{\text}

\newtheorem{theorem}{Theorem}
\newtheorem{lemma}{Lemma}
\newtheorem{cor}{Corollary}
\newtheorem{prop}{Proposition}

\newtheorem*{theorem*}{Theorem}
\newtheorem*{cor*}{Corollary}
\newtheorem*{lemma*}{Lemma}
\newtheorem*{example*}{Example}

\theoremstyle{definition}
\newtheorem{definition}{Definition}
\newtheorem*{def*}{Definition}

\begin{document}
 
 
\title{Prime Parking Functions on Rooted Trees}

\author{Westin King}
%
\author{Catherine H. Yan}
%


\begin{abstract}
For a labeled, rooted tree with edges oriented towards the root, we consider the vertices as parking spots and the edge orientation as a one-way street. Each driver, starting with her preferred parking spot, searches for and parks in the first unoccupied spot along the directed path to the root. If all $n$ drivers park, the sequence of spot preferences is called a parking function. We consider the sequences, called \emph{prime} parking functions, for which each driver parks and each edge in the tree is traversed by some driver after failing to park at her preferred spot. We prove that the total number of prime parking functions on trees with $n$ vertices is $(2n-2)!$. Additionally, we generalize \emph{increasing} parking functions, those in which the drivers park with a weakly-increasing order of preference, to trees and prove that the total number of increasing prime parking functions on trees with $n$ vertices is $(n-1)!S_{n-1}$, where $\{S_i\}_{i \geq 0}$ are the large Schr{\"o}der numbers.
\end{abstract}
\maketitle



\pgfmathsetseed{14285}
\section{Introduction}

Konheim and Weiss first studied classical parking functions in 1966 by examining the probability that a random hashing function would successfully store data when collisions were resolved via linear probing \cite{KONHEIM1966}. Since then, parking functions have appeared in the study of many other combinatorial objects such as noncrossing partitions, hyperplane arrangements, posets, and trees (see \cite{EC2} and \cite{YAN2015}). A straightforward description of a classical parking function $s \in [n]^n$ is that of $n$ drivers, each with a preferred parking spot $s_i$, wishing to park on a one-way street with parking spots labeled $1$ to $n$. One-by-one the drivers attempt to park using the following parking procedure:

\begin{enumerate}
\item Driver $i$ parks at $s_i$ if it is available.
\item If it is not, she parks at the first available spot after $s_i$.
\item If there are none, she leaves the parking lot without parking.
\end{enumerate}

\begin{definition}
The sequence of parking preferences $s = (s_1,s_2,\hdots,s_n)$ is called a \emph{classical parking function of length $n$} if the parking procedure results in all $n$ drivers successfully parking. 
\end{definition}
The sequences $(1,3,2,3,1)$ and $(1,3,4,4,1)$ are parking functions in Figure \ref{fig:5path}. The sequence $(3,3,3,4,5)$ is not, as the spots 1 and 2 are unoccupied after drivers attempt to park. The total number of classical parking functions of length $n$ is $(n+1)^{(n-1)}$ \cite{KONHEIM1966}. 

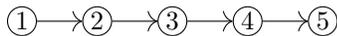
\begin{figure}[h]
\begin{center}
\begin{tikzpicture}
\node[circle,draw,inner sep=1pt] (a) at (0,0) {$1$};
\node[circle,draw,inner sep=1pt] (b) at (1,0) {$2$};
\node[circle,draw,inner sep=1pt] (c) at (2,0) {$3$};
\node[circle,draw,inner sep=1pt] (d) at (3,0) {$4$};
\node[circle,draw,inner sep=1pt] (e) at (4,0) {$5$};

\draw[arrows={->[scale=1.5]}] (a)--(b);
\draw[arrows={->[scale=1.5]}] (b)--(c);
\draw[arrows={->[scale=1.5]}] (c)--(d);
\draw[arrows={->[scale=1.5]}] (d)--(e);
\end{tikzpicture}
\caption{A path of 5 parking spots}
\label{fig:5path}
\end{center}
\end{figure}

It is well-known that a permutation of a classical parking function is still a classical parking function. That is, if $s \in [n]^n$ is a classical parking function and $\sigma \in \mathfrak{S}_n$, then $(s_{\sigma(1)},s_{\sigma(2)},\hdots,s_{\sigma(n)})$ is also a parking function. Those parking functions for which the $s_i$'s are weakly increasing, called \emph{increasing parking functions}, are counted by the ubiquitous Catalan numbers $\{\frac{1}{n+1}{2n \choose n}\}_{n\geq0}$. Rearranging the above examples, $(1,1,2,3,3)$ and $(1,1,3,4,4)$ are both classical increasing parking functions.

Another interesting subset of classical parking functions is defined by the parking functions which, after any 1 is removed from the sequence, are classical parking functions on the first $n-1$ spots. Such parking functions were called \emph{prime} by Gessel and their count is given by $(n-1)^{(n-1)}$ (\cite{EC2}, Exercise 5.49). For example, $(1,3,2,3,1)$ is a prime parking function as both $(3,2,3,1)$ and $(1,3,2,3)$ are parking functions on 4 spots. Notice that any reordering of a prime parking function is also a prime parking function, since rearranging then removing a 1 is the same as removing a 1 and rearranging appropriately.

Many generalizations of parking functions exist, including rational parking functions \cite{ARMSTRONG2016}, $G$-parking functions \cite{POSTNIKOV2004}, $\vec{u}$-parking functions \cite{KUNG2003}, and parking sequences involving cars of different lengths \cite{EHRENBORG2016}. The particular generalization this paper concerns itself with extends the process by which drivers find a place to park and is due to Lackner and Panholzer \cite{LACKNER2016}. Let $T$ be a rooted tree with vertex set $[n]$ and edges oriented towards the root. In this case, we will denote $|T| := n$. Like the classical case, let $s \in [n]^n$ be the sequence of preferred spots. One by one, driver $i$ will attempt to park at $s_i$. If it is already occupied, the driver travels along the unique path towards the root, parking in the first available spot. If no spots are available, then the driver leaves the tree without parking.

\begin{definition}
A sequence $s \in [n]^n$ is a \emph{parking function on $T$} if the parking procedure allows all cars to park on $T$.
\end{definition}

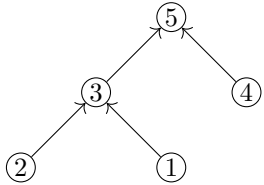
\begin{figure}[h]
\begin{center}
\begin{tikzpicture}
\node[circle,draw,inner sep=1pt] (a) at (0,0) {$5$};
\node[circle,draw,inner sep=1pt] (b) at (-1,-1) {$3$};
\node[circle,draw,inner sep=1pt] (c) at (1,-1) {$4$};
\node[circle,draw,inner sep=1pt] (d) at (-2,-2) {$2$};
\node[circle,draw,inner sep=1pt] (e) at (0,-2) {$1$};

\draw[arrows={->[scale=1.2]}] (b)--(a);
\draw[arrows={->[scale=1.2]}] (c)--(a);
\draw[arrows={->[scale=1.2]}] (d)--(b);
\draw[arrows={->[scale=1.2]}] (e)--(b);
\end{tikzpicture}
\caption{ A tree with parking function $(2,2,1,4,2)$.}
\label{fig:5tree}
\end{center}
\end{figure}

We call the pair $(T,s)$ a \emph{parking function} and give an example in Figure \ref{fig:5tree}. Notice that tree parking functions are classical parking functions when $T$ is a path with vertices labeled in decreasing order away from root $n$. We denote by $\mathcal{P}_n$ the path with $n$ vertices on which classical parking functions are defined. Figure \ref{fig:5path} is a picture of $\mathcal{P}_5$. We also note that, because vertices have outdegree at most one, the parking procedure is deterministic on rooted trees and thus well-defined.

Let $F_n$ be the total number of pairs $(T,s)$ for which $|T|=n$ and set $F(x) = \sum\limits_{n\geq 1}F_n\dfrac{x^n}{(n!)^2}$. Lackner and Panholzer \cite{LACKNER2016} found that $F(x)$ satisfies the equation

\begin{equation}\label{eq:recursion}
F(x) = T(2x)+ \ln\lp1-\frac{T(2x)}{2}\rp,
\end{equation}
where $T(x) = \sum\limits_{n\geq 1}n^{n-1}\dfrac{x^n}{n!}$ is the tree function. Furthermore, they calculated

\[
F_n = ((n-1)!)^2 \cdot \lp\sum\limits_{i=0}^{n-1}\frac{(n-i)\cdot (2n)^i}{i!}\rp.
\]

In Section \ref{sec:prime}, we define a generalization of classical prime parking functions to trees. In Section \ref{sec:gfproof}, we decompose a parking function $(T,s)$ and use Equation \eqref{eq:recursion} to prove the following theorem.
\begin{theorem}\label{thm:main}
The total number of prime tree parking functions $P_n$ for $n \geq 1$ is given by
\[
P_n = (2n-2)!
\]
\end{theorem}
In Section \ref{sec:bijection}, we give a bijective proof for Theorem \ref{thm:main} by matching a prime parking function $(T,p)$ with a permutation and a rooted, ordered (plane) tree with labeled non-root vertices. In Section \ref{sec:extras}, we consider special subsets of the domain and image of a bijection constructed in Section \ref{sec:bijection}.

In Section \ref{sec:PD}, we consider a natural generalization of increasing parking functions, called \emph{parking distributions}, first studied by Butler, Graham, and Yan \cite{BUTLER2017}.

\begin{definition}
A \emph{parking distribution} is a parking function $(T,s)$ such that $s$ is weakly increasing.
\end{definition}
Like the classical case, the order in which drivers attempt to park does not affect their ability to park. As such, the name ``distribution" is chosen because the distribution of drivers, $|\{j : s_j = i\}|$ for $i \in V(T)$, determines whether all drivers successfully park. In this paper, we let $s$ be weakly increasing for convenience. We discuss the relationship between regular and prime parking distributions and give a generating function proof of the following theorem.

\begin{restatable}{theorem}{primedistr}\label{thm:primedistr}
The total number of prime parking distributions on trees with $n \geq 1$ vertices is given by
\[
\widetilde{P}_n = (n-1)!S_{n-1},
\]
where $\{S_i\}_{i\geq 0}$ are the large Schr{\"o}der numbers.
\end{restatable}
Finally, in Section \ref{sec:conclusion}, we discuss some directions for further research. This paper is the full version of Section 3 of \cite{KINGFPSAC2018}.

\section{Prime Parking Functions} \label{sec:prime}

Let $T$ be a rooted tree with vertex set $[n]$ and $v \in [n]$. We define a preorder $\preceq_T$ on $[n]$ by letting $v \preceq_T w$ if there exists a directed path from $v$ to $w$ in $T$. By convention we say $v \preceq_T v$. Let $v \in [n]$ and define $T_v$ to be the subgraph induced by the set of vertices $\{u : u \preceq_T v\}$. The following characterization of a parking function is a slight modification of that given by Lackner and Panholzer \cite{LACKNER2016}.

\begin{prop}\label{prop:prime}
The pair $(T,s)$ is a parking function if and only if for all $v \in [n]$, we have $|T_v| \leq |\{i : s_i \in T_v\}|$.
\end{prop}

\begin{proof}
Suppose $(T,s)$ is a parking function. Then at the end of parking, every vertex is occupied by some car. For any $v$, the only drivers that can park in $T_v$ are those preferring a spot in $T_v$, so we must have $|T_v| \leq |\{i : s_i \in T_v\}|$.

On the other hand, suppose that $(T,s)$ is not a parking function. Then there is some driver who does not park, meaning there is some unoccupied spot $v$. Then $|T_v| > |\{i : s_i \in T_v\}|$, since $v$ is empty and drivers must park in the first unoccupied spot they reach.
\end{proof}
Notice in the classical parking function case, $(\mathcal{P}_n,s)$ is a parking function if and only if for $k \in [n]$, we have $k \leq |\{i: s_i \leq k\}|$. In addition, a classical parking function $s$ is prime if and only if $k < |\{i: s_i \geq k\}|$ when $ k \leq n-1$ (see \cite{EC2}, Exercise 5.49f). In this spirit, we define prime parking functions on trees and give an example in Figure \ref{fig:primetree}.

\begin{definition}
A parking function $(T,s)$ is \emph{prime} if, for every non-root vertex $v \in [n]$, we have $|T_v| < |\{i : s_i \in T_v\}|$.
\end{definition}

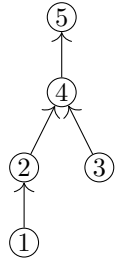
\begin{figure}[h!]
\begin{center}
\begin{tikzpicture}
\node[circle,draw,inner sep=1pt] (a) at (0,0) {$5$};
\node[circle,draw,inner sep=1pt] (b) at (0,-1) {$4$};
\node[circle,draw,inner sep=1pt] (c) at (0.5,-2) {$3$};
\node[circle,draw,inner sep=1pt] (d) at (-0.5,-2) {$2$};
\node[circle,draw,inner sep=1pt] (e) at (-0.5,-3) {$1$};

\draw[arrows={->[scale=1.5]}] (b)--(a);
\draw[arrows={->[scale=1.5]}] (d)--(b);
\draw[arrows={->[scale=1.5]}] (c)--(b);
\draw[arrows={->[scale=1.5]}] (e)--(d);
\end{tikzpicture}
\caption{$T$ with prime $p = (1,3,2,3,1)$.}
\label{fig:primetree}
\end{center}
\end{figure} 

We briefly introduce some notation. If there is an edge of the form $u \rightarrow v$, we represent the edge as the ordered pair $(u,v)$. Additionally, we denote $$\mathcal{PF}_n = \{(T,p) : |T|=n \text{ and $p$ is a prime parking function on }T\},$$ and so $P_n = |\mathcal{PF}_n|$. Prime parking functions can also be understood in terms of edges crossed by cars failing to park at their preferred spots. 

\begin{definition}
For a parking function $(T,s)$, we say that an edge $e$ is \emph{used} by $s$ if there exists some driver who, after failing to park at her preferred spot, crosses $e$ during her search for an unoccupied spot. 
\end{definition}
\begin{figure}[h]
\begin{center}
\begin{tikzpicture}
\node[circle,draw,inner sep=1pt] (a) at (0,0) {$1$};
\node[circle,draw,inner sep=1pt] (b) at (1,0) {$2$};
\node[circle,draw,inner sep=1pt] (c) at (2,0) {$3$};
\node[circle,draw,inner sep=1pt] (d) at (3,0) {$4$};
\node[circle,draw,inner sep=1pt] (e) at (4,0) {$5$};

\draw[arrows={->[scale=1.5]}] (a)--(b);
\draw[arrows={->[scale=1.5]},dashed] (b)--(c);
\draw[arrows={->[scale=1.5]},dashed] (c)--(d);
\draw[arrows={->[scale=1.5]}] (d)--(e);
\end{tikzpicture}
\caption{Solid edges are used by $s=(1,3,4,4,1)$}
\label{fig:crossed}
\end{center}
\end{figure}
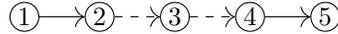
Figure \ref{fig:crossed} shows a classical parking function with both used and unused edges. Whether or not an edge $(u,v)$ is used by a parking function depends only on the number of cars preferring $T_v$.

\begin{prop}\label{prop:edges}
Given a parking function $(T,s)$, an edge $e = (u,v)$ is used by $s$ if and only if $|T_u| < |\{ i : s_i \in T_u\}|$. Furthermore, the set of edges used by $s$ is invariant under permutations of $s$.
\end{prop}
\begin{proof}
Suppose $e = (u,v)$ is used by $s$. Then at least one driver preferring $T_u$ does not park in $T_u$. No cars preferring a vertex outside $T_u$ can park inside, as $T_u$ consists of all vertices $w$ such that $w \preceq_T u$. The pair $(T,s)$ is a parking function, so it follows that $|T_u|< |\{ i : s_i \in T_u\}|$. On the other hand, if $|T_u|< |\{ i : s_i \in T_u\}|$, then as $s$ is a parking function on $T$, at least one driver preferring $T_u$ must park outside. This driver must cross $e$ in order to do so.

The set characterization of used edges does not change if the letters of $s$ are permuted.
\end{proof}
As an immediate corollary we see,

\begin{cor}\label{cor:edges_prime}
A parking function $(T,p)$ is prime if and only if every edge in $T$ is used by $p$.
\end{cor}

%
%
%
%

\section{Proof of Theorem \ref{thm:main} via Generating Functions} \label{sec:gfproof}

Recall that $F_n$ is the number of pairs $(T,s)$ with $|T|=n$ and $s$ a parking function on $T$, while $P_n$ is the number of such pairs where $s$ is prime. We define for both $F_n$ and $P_n$ the generating functions
\[
F(x) = \sum\limits_{n \geq 1} F_n \dfrac{x^n}{(n!)^2} \hspace{1cm} \tx{and}  \hspace{1cm} P(x) = \sum\limits_{n \geq 1} P_n \dfrac{x^n}{(n!)^2}.
\]
We choose $(n!)^2$ for the denominator to account for both the relabellings of the tree's vertices and the re-orderings of the preference sequence $s$.

For a parking function $(T,s)$, we consider a decomposition of $T$ into a ``core" component supporting a prime parking function, $(T_0,s^{(0)})$, and some collection of general parking functions, $(T_i,s^{(i)})$, attached to the core component. Let $T_0$ be the subtree of $T$ containing the root and all vertices connected to the root via edges used by $s$. Let $s^{(0)}$ be the subsequence of $s$ defined by drivers preferring vertices in $T_0$. Notice that $s^{(0)}$ is a prime parking function by construction. The other subtrees $T_i$ are the connected components remaining after deleting edges $(u,v)$ unused by $s$, where $v \in V(T_0)$ and $u \notin V(T_0)$, and $s^{(i)}$ are the subsequences of $s$ consisting of drivers preferring vertices in $T_i$. 

Figure \ref{fig:generaldecomp} gives a general overview of this decomposition. Dashed edges are those unused by $s$ but connected to the ``core" component $T_0$. In Figure \ref{fig:crossed}, the ``core" component is the subtree induced by vertex set $\{4,5\}$ while $s^{(0)} = (4,4)$. The one other component is the subgraph induced by vertices $\{1,2,3\}$ (identical to $\mathcal{P}_3$) with $s^{(1)} = (1,3,1)$. Notice that $(1,3,1)$ is not a prime parking function on $\mathcal{P}_3$.

\begin{figure}[h!]
\begin{center}
\begin{tikzpicture}
\node[ellipse,draw, minimum width=90pt, minimum height=65pt] at (0,0) {$T_0$};
\node[ellipse,draw, minimum width=45pt, minimum height=70pt] at (-3.2,-2.5) {$T_1$};
\node[ellipse,draw, minimum width=45pt, minimum height=70pt] at (-1.5,-2.5) {$T_2$};
\node[ellipse,draw, minimum width=45pt, minimum height=70pt] at (0.2,-2.5) {$T_3$};
\node[ellipse,draw, minimum width=45pt, minimum height=70pt] at (2.5,-2.5) {$T_r$};
\node at (1.3,-2.5) {$\hdots$};

\node at (-3.2,-1.77){
\scalebox{.3}{
\begin{forest}
random tree/.style n args={2}{
if={#1>0}{repeat={random(0,#2)}{append={[,random tree={#1-1}{#2}]}}}{},
  parent anchor=center, child anchor=center, grow=south},
[,random tree={3}{3}]
\end{forest}
}
};
\node at (-1.5,-1.77){
\scalebox{.3}{
\begin{forest}
random tree/.style n args={2}{
if={#1>0}{repeat={random(0,#2)}{append={[,random tree={#1-1}{#2}]}}}{},
  parent anchor=center, child anchor=center, grow=south},
[,random tree={3}{3}]
\end{forest}
}
};
\node at (0.2,-1.77){
\scalebox{.3}{
\begin{forest}
random tree/.style n args={2}{
if={#1>0}{repeat={random(0,#2)}{append={[,random tree={#1-1}{#2}]}}}{},
  parent anchor=center, child anchor=center, grow=south},
[,random tree={3}{3}]
\end{forest}
}
};
\node at (2.7,-1.77){
\scalebox{.3}{
\begin{forest}
random tree/.style n args={2}{
if={#1>0}{repeat={random(0,#2)}{append={[,random tree={#1-1}{#2}]}}}{},
  parent anchor=center, child anchor=center, grow=south},
[,random tree={3}{3}]
\end{forest}
}
};
\node at (0,0.7){
\scalebox{.4}{
\begin{forest}
random tree/.style n args={2}{
if={#1>0}{repeat={random(0,#2)}{append={[,random tree={#1-1}{#2}]}}}{},
  parent anchor=center, child anchor=center, grow=south},
[,random tree={2}{3}]
\end{forest}
}
};

\draw[arrows={->[scale=1.5]},dashed] (-3.2,-1.28)--(-1,-.5);
\draw[arrows={->[scale=1.5]},dashed] (-1.5,-1.28)--(-.3,0);
\draw[arrows={->[scale=1.5]},dashed] (0.2,-1.28)--(0.5,0.5);
\draw[arrows={->[scale=1.5]},dashed] (2.5,-1.28)--(0.5,0.5);

\end{tikzpicture}
\caption{Decomposition into components.}
\label{fig:generaldecomp}
\end{center}
\end{figure}
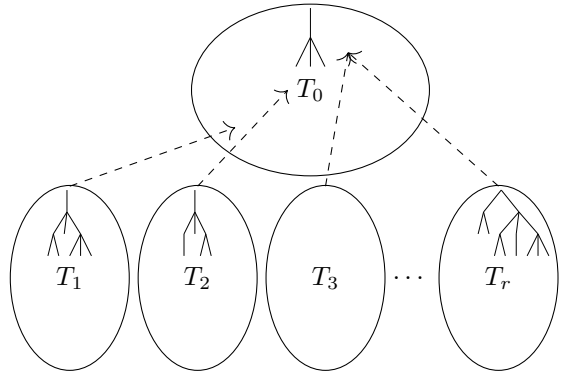

In this way, we can construct any parking function $(T,s)$ with $|T|=n$ by choosing a prime parking function $(T_0,s^{(0)})$ with $|T_0|=k_0$ and $r$-many other regular parking functions $\{(T_i,s^{(i)})\}_{i=1}^r$ with $k_i  = |T_i|  \geq 1$ and $\sum_{i=0}^r k_i = n$. From there, we can attach each $T_i$ to $T_0$ in one of $k_0$-many places to form $T$. What remains is to choose the labels on $T$ and choose which indices in $s$ each $s^{(i)}$ is assigned. The $1/r!$ accounts for the order in which the $T_i$ are chosen and attached. This means

\[
F_n = \sum\limits_{r\geq 0}\frac{1}{r!}\sum\limits_{\sum\limits_{i=0}^rk_i=n}P_{k_{0}} F_{k_{1}}\cdots F_{k_{r}}{n \choose k_0,k_1,\hdots,k_r}^2(k_0)^r.
\]
Since $F_1=1$, summing over $n \geq 1$, we get the relationship

\begin{equation}\label{eq:composition}
F(x) = P\lp xe^{F(x)}\rp.
\end{equation}
Using \eqref{eq:recursion},

\begin{equation}\label{eq:Pcombined}
P\lp xe^{F(x)} \rp = T(2x) + \ln\lp 1-\frac{T(2x)}{2} \rp.
\end{equation}
Setting $z=z(x) = xe^{F(x)}$, $y=y(x) = \frac{T(2x)}{2}$, and using the relation $T(x) = xe^{T(x)}$, we notice from Equation \eqref{eq:recursion} that 
\[
z = y(1-y).
\] 
Solving the quadratic equation gives

\[
y = zC(z),
\]
where $C(x) = \sum\limits_{n \geq 0}C_n x^n$ is the ordinary generating function for the Catalan numbers with analytic expression

\begin{equation*}
C(x) = \frac{1-\sqrt{1-4x}}{2x}.
\end{equation*}
Since $z(0) = 0$ and $z_1\neq 0$, the formal power series $z(x)$ has a compositional inverse. Rewriting Equation \eqref{eq:Pcombined} and plugging in the inverse, we get

\[
P(x) = 2xC(x) + \ln(1-xC(x)).
\]
The Catalan generating function $C(x)$ satisfies the recursion $C(x) = 1+xC(x)^2$ and thus $C'(x)$ satisfies

\begin{equation*}
C'(x) = \frac{xC(x)^2}{1-2xC(x)}.
\end{equation*}
After some algebra, we can see

\begin{equation}\label{eq:oddcatalan}
\frac{C(x)}{(xC(x))'} = \frac{1-2xC(x)}{1-xC(x)}.
\end{equation}
Then taking the derivative of $P(x)$ and using \eqref{eq:oddcatalan}, we have

\begin{align*}
P'(x) &= 2(xC(x))' - \frac{(xC(x))'}{1-xC(x)} \\
&= (xC(x))'\lp\frac{1-2xC(x)}{1-xC(x)} \rp \\
&= C(x).
\end{align*}
Therefore,

\[
P(x) = \sum\limits_{n\geq1}\frac{C_{n-1}}{n}x^n = \sum\limits_{n\geq 1}(2n-2)!\frac{x^n}{(n!)^2}.
\]
Hence, $P_n = (2n-2)!$ as claimed. Such a simple number demands a bijective proof, which we give in the next section.

\section{Bijective Proof of Theorem \ref{thm:main}}\label{sec:bijection}

In order to determine $P_n$, we define a bijection $\psi: (T,p) \mapsto (\sigma,P)$ where $\sigma \in \mathfrak{S}_n$ and $P$ is an ordered tree with $n$ vertices whose non-root vertices are labeled by $[n-1]$. An ordered tree, also called a plane tree, is a rooted tree for which siblings are given a linear order. 

To describe $\psi$, we first give a definition. Recall post-order labeling: given an ordered tree $\mathcal{T}$, travel around the left border of the tree starting from the root, labeling in increasing order as one reaches the right side of a vertex. See the final picture in Figure \ref{fig:phi} for a tree labeled by post-order.

\begin{definition}\label{def:SRPn}
For an ordered tree $\mathcal{T}$ with $|\mathcal{T}|=n$ and sequence $p \in [n]^n$, we say the pair $(\mathcal{T},p)$ is a \emph{standardized restricted prime parking function} if 
\begin{enumerate}
\item $(\mathcal{T},p)$ is a prime parking function when $\mathcal{T}$ is considered an unordered tree.
\item For any pair of sibling vertices, $\{u,v\}\subseteq V(\mathcal{T})$, the vertex $v$ is ordered to the right of $u$ if and only if the edge $(v,w)$, where $w$ is the parent vertex, is crossed by some driver before the edge $(u,w)$ is crossed during the parking procedure.
\item $\mathcal{T}$ is labeled via post-order.
\end{enumerate}
\end{definition}

We denote the set of standardized restricted prime parking functions on $n$ vertices by $\mathcal{SRP}_n$. If $\mathcal{T}$ is an ordered tree, we say $(\mathcal{T},p)$ is a prime parking function when the pair is a parking function if the ordering on $\mathcal{T}$ is forgotten.

The bijection $\psi$ is constructed from two bijections, $\phi: (T,p) \mapsto (\sigma,(T_\sigma,p_\sigma))$ where $(T_\sigma,p_\sigma) \in \mathcal{SRP}_n$, and $\alpha: (T_\sigma,p_\sigma) \mapsto P$, where $P$ is the aforementioned ordered tree with non-root vertices labeled by $[n-1]$. The unlabeled ordered trees on $n$ vertices are counted by the Catalan number $C_{n-1}$, so including the labellings there are $C_{n-1}(n-1)!$-many such $P$. Combining these two steps, we conclude that $P_n = n!C_{n-1}(n-1)! = (2n-2)!$.

\subsection{The Bijection $\phi$}

We prove

\begin{prop}\label{prop:SRPn}
For $n \geq 1$, we have
\[
P_n = n!|\mathcal{SRP}_n|.
\]
\end{prop}

\begin{proof}
Let $(T,p) \in \mathcal{PF}_n$. We induce an ordering on $T$ by using $p$. Since the parking function is prime, every edge must be crossed by some driver after failing to park at her preferred spot. Additionally, since there is only one path a driver may travel in search of a spot and because drivers must park at the first available spot they encounter, the order in which edges are initially crossed is well-defined. Then, consider $T$ as an ordered tree by ordering vertices to match property 2 in Definition \ref{def:SRPn}.

For a $\sigma \in \mathfrak{S}_n$, let $T_\sigma$ be the tree obtained by relabeling the vertices of $T$, $v \mapsto \sigma(v)$. Likewise, let $p_\sigma = (\sigma(p_1),\sigma(p_2),\hdots,\sigma(p_n))$. Then we define
\[
\phi((T,p)) = (\sigma,(T_\sigma,p_\sigma)),
\]
where $\sigma$ is the unique permutation such that $T_\sigma$ is labeled by post-order. This means $(T_\sigma,p_\sigma) \in \mathcal{SRP}_n$.

Since any two relabellings of an ordered tree are distinct, for $(\widetilde{T},\widetilde{p}) \in \mathcal{SRP}_n$, the preimage $\phi^{-1}((\sigma,(\widetilde{T},\widetilde{p}))$ is the prime parking function $(\widetilde{T}_{\sigma^{-1}},\widetilde{p}_{\sigma^{-1}})$, where the ordering on the tree is forgotten.
\end{proof}

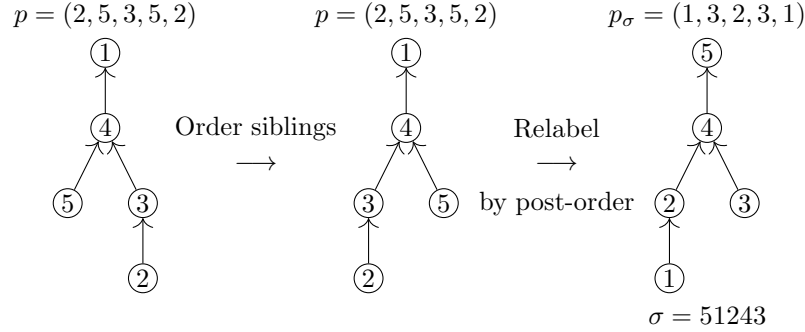
\begin{figure}[h]
\begin{center}
\begin{tikzpicture}
\node[circle,draw,inner sep=1pt] (a) at (0,0) {$1$};
\node[circle,draw,inner sep=1pt] (b) at (0,-1) {$4$};
\node[circle,draw,inner sep=1pt] (c) at (-.5,-2) {$5$};
\node[circle,draw,inner sep=1pt] (d) at (.5,-2) {$3$};
\node[circle,draw,inner sep=1pt] (e) at (.5,-3) {$2$};
\draw[arrows={->[scale=1.5]}] (b)--(a);
\draw[arrows={->[scale=1.5]}] (c)--(b);
\draw[arrows={->[scale=1.5]}] (d)--(b);
\draw[arrows={->[scale=1.5]}] (e)--(d);
\node (z) at (0,0.5) {$p = (2,5,3,5,2)$};

\node[circle,draw,inner sep=1pt] (f) at (4,0) {$1$};
\node[circle,draw,inner sep=1pt] (g) at (4,-1) {$4$};
\node[circle,draw,inner sep=1pt] (h) at (4.5,-2) {$5$};
\node[circle,draw,inner sep=1pt] (i) at (3.5,-2) {$3$};
\node[circle,draw,inner sep=1pt] (j) at (3.5,-3) {$2$};
\draw[arrows={->[scale=1.5]}] (g)--(f);
\draw[arrows={->[scale=1.5]}] (h)--(g);
\draw[arrows={->[scale=1.5]}] (i)--(g);
\draw[arrows={->[scale=1.5]}] (j)--(i);
\node (k) at (4,0.5) {$p = (2,5,3,5,2)$};

\node[circle,draw,inner sep=1pt] (l) at (8,0) {$5$};
\node[circle,draw,inner sep=1pt] (m) at (8,-1) {$4$};
\node[circle,draw,inner sep=1pt] (n) at (8.5,-2) {$3$};
\node[circle,draw,inner sep=1pt] (o) at (7.5,-2) {$2$};
\node[circle,draw,inner sep=1pt] (p) at (7.5,-3) {$1$};
\draw[arrows={->[scale=1.5]}] (m)--(l);
\draw[arrows={->[scale=1.5]}] (n)--(m);
\draw[arrows={->[scale=1.5]}] (o)--(m);
\draw[arrows={->[scale=1.5]}] (p)--(o);
\node (q) at (8,0.5) {$p_\sigma = (1,3,2,3,1)$};
\node (t) at (8,-3.5) {$\sigma = 51243$};

\node (r) at (2,-1.5) {$\longrightarrow$};
\node (s) at (6,-1.5) {$\longrightarrow$};
\node (u) at (2,-1) {Order siblings};
\node (v) at (6,-1) {Relabel};
\node (w) at (6,-2) {by post-order};
\end{tikzpicture}
\caption{An example of $\phi$.}
\label{fig:phi}
\end{center}
\end{figure}

Figure \ref{fig:phi} shows an application of $\phi$. In the left tree, the edge $(5,4)$ is first used by the fourth driver, while the edge $(3,4)$ is not used until the fifth driver. Thus, the vertex 5 is placed to the right of vertex 3, which gives the ordered tree in the middle of the figure. The right tree is obtained from the middle one by relabeling via post-order, which gives the permutation $\sigma$. We now turn our attention to $|\mathcal{SRP}_n|$.

\subsection{The Bijection $\alpha$}\label{subsec:alpha}

We define $\alpha: \mathcal{SRP}_n \rightarrow \{\text{ordered trees with non-root vertices labeled by }[n-1]\}$. The main observation necessary for the construction of $\alpha$ is a decomposition of $(\mathcal{T},p) \in \mathcal{SRP}_n$ into an ordered collection of components, each a relabeling of a standardized restricted prime parking function, based on the edges used by the first $n-1$ drivers. We use $\mathcal{T}$ to emphasize that $\mathcal{T}$ is an ordered tree. First, we give a proposition about where the final driver must park for prime parking functions, which also applies to parking functions in $\mathcal{SRP}_n$.

\begin{prop}\label{prop:finalspot}
Let $(T,p) \in \mathcal{PF}_n$. Then the final driver parks at the root node.
\end{prop}

\begin{proof}
Let $\omega$ be the vertex the final driver parks at. If $\omega$ is not the root, it has a parent vertex $v$. Since drivers must park at the first empty vertex they arrive at, the edge $(\omega,v)$ can not be used by any driver prior to the final one, since $\omega$ is unoccupied. Since the final driver also does not use $(\omega,v)$, it remains unused. However, $p$ is a prime parking function on $T$, so all edges must be used. Thus, there can be no edge $(\omega,v)$, and so $\omega$ is the root of $T$.
\end{proof}

We first describe the recursive construction of $\alpha$, then prove it is a bijection. Begin with $(\mathcal{T},p) \in \mathcal{SRP}_n$. We use the tree in Figure \ref{fig:running} as a running example.

\begin{figure}[h]
\begin{center}
\begin{tikzpicture}
\node (a) at (0,0.5) {$p=(6,4,1,3,3,1,6,7,2)$};
\node[circle,draw,inner sep=1pt] (b) at (0,0) {$9$};
\node[circle,draw,inner sep=1pt] (c) at (0,-.75) {$8$};
\node[circle,draw,inner sep=1pt] (d) at (0.5,-1.5) {$7$};
\node[circle,draw,inner sep=1pt] (e) at (0.5,-2.25) {$6$};
\node[circle,draw,inner sep=1pt] (f) at (-0.5,-1.5) {$5$};
\node[circle,draw,inner sep=1pt] (g) at (-0.5,-2.25) {$4$};
\node[circle,draw,inner sep=1pt] (h) at (-0.5,-3) {$3$};
\node[circle,draw,inner sep=1pt] (i) at (-0.5,-3.75) {$2$};
\node[circle,draw,inner sep=1pt] (j) at (-0.5,-4.5) {$1$};

\draw[arrows={->[scale=1.5]}] (c)--(b);
\draw[arrows={->[scale=1.5]}] (d)--(c);
\draw[arrows={->[scale=1.5]}] (e)--(d);
\draw[arrows={->[scale=1.5]}] (f)--(c);
\draw[arrows={->[scale=1.5]}] (g)--(f);
\draw[arrows={->[scale=1.5]}] (h)--(g);
\draw[arrows={->[scale=1.5]}] (i)--(h);
\draw[arrows={->[scale=1.5]}] (j)--(i);
\end{tikzpicture}
\caption[caption]{A $(\mathcal{T},p) \in \mathcal{SRP}_n$.}
\label{fig:running}
\end{center}
\end{figure}
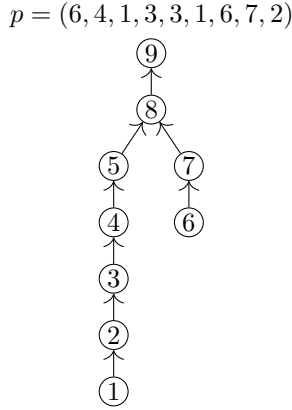

\textbf{Base Case.} If $\mathcal{T}$ is a singleton, then $\alpha((\mathcal{T},p))$ is an unlabeled singleton as $|\mathcal{SRP}_1|=1$.

\textbf{Step 1.} Park all except the final driver, highlighting edges as they are used. Delete the non-highlighted edges and the root, marking the terminal vertex of any edge deleted and the vertex with label $p_n$.

Since $p$ is a prime parking function on $\mathcal{T}$, the non-highlighted edges must lie on the path $\mathcal{P}$ between the vertex labeled $p_n$ and the root. The highlighted edges define some collection of subtrees $\{\mathcal{T}_i\}_{i=1}^r$, linearly ordered by the order in which they are a part of $\mathcal{P}$. Since the root is always isolated, as the final driver is the only one to cross the edge connected to the root, we may ignore it. Let $p^{(i)}$ be the subsequence of $(p_1,\hdots,p_{n-1})$ consisting of all $p_j$ such that $p_j \in V(\mathcal{T}_i)$. By construction, $p^{(i)}$ is a prime parking function on $\mathcal{T}_i$ satisfying conditions 1 and 2 in Definition \ref{def:SRPn}. See Figure \ref{fig:decomp} for this step applied to our running example. The non-highlighted edges are dotted and the marked vertices are shaded.

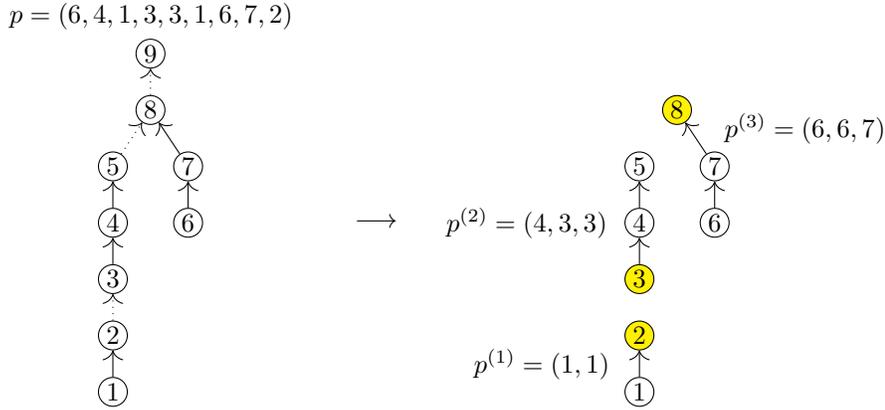
\begin{figure}[h]
\begin{center}
\begin{tikzpicture}
\node (a) at (0,0.5) {$p=(6,4,1,3,3,1,6,7,2)$};
\node[circle,draw,inner sep=1pt] (b) at (0,0) {$9$};
\node[circle,draw,inner sep=1pt] (c) at (0,-.75) {$8$};
\node[circle,draw,inner sep=1pt] (d) at (0.5,-1.5) {$7$};
\node[circle,draw,inner sep=1pt] (e) at (0.5,-2.25) {$6$};
\node[circle,draw,inner sep=1pt] (f) at (-0.5,-1.5) {$5$};
\node[circle,draw,inner sep=1pt] (g) at (-0.5,-2.25) {$4$};
\node[circle,draw,inner sep=1pt] (h) at (-0.5,-3) {$3$};
\node[circle,draw,inner sep=1pt] (i) at (-0.5,-3.75) {$2$};
\node[circle,draw,inner sep=1pt] (j) at (-0.5,-4.5) {$1$};

\draw[arrows={->[scale=1.5]},dotted] (c)--(b);
\draw[arrows={->[scale=1.5]}] (d)--(c);
\draw[arrows={->[scale=1.5]}] (e)--(d);
\draw[arrows={->[scale=1.5]},dotted] (f)--(c);
\draw[arrows={->[scale=1.5]}] (g)--(f);
\draw[arrows={->[scale=1.5]}] (h)--(g);
\draw[arrows={->[scale=1.5]},dotted] (i)--(h);
\draw[arrows={->[scale=1.5]}] (j)--(i);

\node (v) at (3,-2.25) {$\longrightarrow$};

\node (k) at (5.2,-4.13) {$p^{(1)}=(1,1)$};
\node[circle,draw,inner sep=1pt,fill=yellow] (l) at (6.5,-3.75) {$2$};
\node[circle,draw,inner sep=1pt] (m) at (6.5,-4.5) {$1$};

\draw[arrows={->[scale=1.5]}] (m)--(l);

\node (n) at (5,-2.25) {$p^{(2)}=(4,3,3)$};
\node[circle,draw,inner sep=1pt] (o) at (6.5,-1.5) {$5$};
\node[circle,draw,inner sep=1pt] (p) at (6.5,-2.25) {$4$};
\node[circle,draw,inner sep=1pt,fill=yellow] (q) at (6.5,-3) {$3$};

\draw[arrows={->[scale=1.5]}] (p)--(o);
\draw[arrows={->[scale=1.5]}] (q)--(p);

\node (r) at (8.7,-1) {$p^{(3)}=(6,6,7)$};
\node[circle,draw,inner sep=1pt,fill=yellow] (s) at (7,-.75) {$8$};
\node[circle,draw,inner sep=1pt] (t) at (7.5,-1.5) {$7$};
\node[circle,draw,inner sep=1pt] (u) at (7.5,-2.25) {$6$};

\draw[arrows={->[scale=1.5]}] (t)--(s);
\draw[arrows={->[scale=1.5]}] (u)--(t);
\end{tikzpicture}
\caption[caption]{Step 1, decomposing.}
\label{fig:decomp}
\end{center}
\end{figure}

\textbf{Step 2.} For each $(\mathcal{T}_i,p^{(i)})$, let $A_i =\{j \in [n-1] : p_j \in V(\mathcal{T}_i)\}$. For $1 \leq i \leq r$, if the marked vertex on $\mathcal{T}_i$ has the $k^\tx{th}$ smallest label among vertices in $\mathcal{T}_i$, mark the $k^\tx{th}$ smallest element in $A_i$.

Notice that the elements of $A_i$ are precisely those $j$ such that $p_j$ appears in $p^{(i)}$. The marked vertices and elements track how the $\mathcal{T}_i$ are connected to each other in $\mathcal{T}$. Figure \ref{fig:labelchildren} shows this step. The elements of $A_i$ are represented by $C_j$ instead of $j$ to emphasize that $j$ is an index from $p$, rather than a vertex of $\mathcal{T}$.

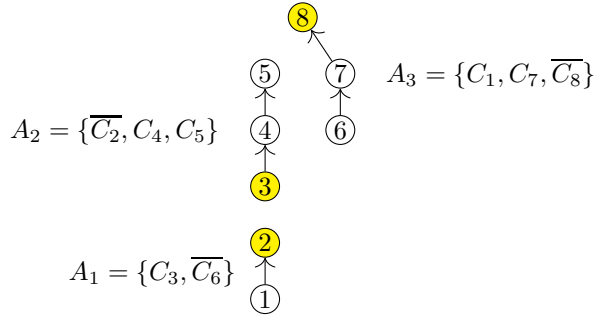
\begin{figure}[h]
\begin{center}
\begin{tikzpicture}
\node (a) at (0,0.5) {$p=(6,4,1,3,3,1,6,7,2)$};
\node[circle,draw,inner sep=1pt,fill=yellow] (c) at (0,-.75) {$8$};
\node[circle,draw,inner sep=1pt] (d) at (0.5,-1.5) {$7$};
\node[circle,draw,inner sep=1pt] (e) at (0.5,-2.25) {$6$};
\node[circle,draw,inner sep=1pt] (f) at (-0.5,-1.5) {$5$};
\node[circle,draw,inner sep=1pt] (g) at (-0.5,-2.25) {$4$};
\node[circle,draw,inner sep=1pt,fill=yellow] (h) at (-0.5,-3) {$3$};
\node[circle,draw,inner sep=1pt,fill=yellow] (i) at (-0.5,-3.75) {$2$};
\node[circle,draw,inner sep=1pt] (j) at (-0.5,-4.5) {$1$};

\draw[arrows={->[scale=1.5]}] (d)--(c);
\draw[arrows={->[scale=1.5]}] (e)--(d);
\draw[arrows={->[scale=1.5]}] (g)--(f);
\draw[arrows={->[scale=1.5]}] (h)--(g);
\draw[arrows={->[scale=1.5]}] (j)--(i);

\node (p) at (-2,-4.13) {$A_1=\{C_3,\overline{C_6}\}$};
\node (q) at (-2.5,-2.25) {$A_2 = \{\overline{C_2},C_4,C_5\}$};
\node (r) at (2.5,-1.5) {$A_3 = \{C_1,C_7,\overline{C_8}\}$};
\end{tikzpicture}
\caption[caption]{Step 2, tracking $p$ and the shape of $\mathcal{T}$.}
\label{fig:labelchildren}
\end{center}
\end{figure}

\textbf{Step 3.} Relabel each $(\mathcal{T}_i,p^{(i)})$ so that its labels are in post-order, meaning $(\overline{\mathcal{T}}_i,\overline{p}^{(i)}) \in \mathcal{SRP}_n$. Apply $\alpha$ to each $(\overline{\mathcal{T}}_i,\overline{p}^{(i)})$.

In Figure \ref{fig:induction}, we show components after they have been relabeled and apply $\alpha$ to each.

\begin{figure}[h]
\begin{center}
\begin{tikzpicture}
\node (k) at (0,1) {$\bar{p}^{(1)}=(1,1)$};
\node[circle,draw,inner sep=1pt] (l) at (0,0.5) {$2$};
\node[circle,draw,inner sep=1pt] (m) at (0,-.5) {$1$};
\draw[arrows={->[scale=1.5]}] (m)--(l);

\node (n) at (2,0.5) {$\bar{p}^{(2)}=(2,1,1)$};
\node[circle,draw,inner sep=1pt] (o) at (2,0) {$3$};
\node[circle,draw,inner sep=1pt] (p) at (2,-1) {$2$};
\node[circle,draw,inner sep=1pt] (q) at (2,-2) {$1$};
\draw[arrows={->[scale=1.5]}] (p)--(o);
\draw[arrows={->[scale=1.5]}] (q)--(p);

\node (r) at (4.2,1) {$\bar{p}^{(3)}=(1,1,2)$};
\node[circle,draw,inner sep=1pt] (s) at (4,0.5) {$3$};
\node[circle,draw,inner sep=1pt] (t) at (4,-.5) {$2$};
\node[circle,draw,inner sep=1pt] (u) at (4,-1.5) {$1$};
\draw[arrows={->[scale=1.5]}] (t)--(s);
\draw[arrows={->[scale=1.5]}] (u)--(t);

\node (v) at (5.5,-.5) {$\longrightarrow$};

\node (a) at (7,1) {$\alpha\lp(\overline{\mathcal{T}}_1,\bar{p}^{(1)}) \rp$};
\node[circle,draw,radius=1cm] (b) at (7,0.5) {};
\node[circle,draw,inner sep=1pt] (c) at (7,-.5) {$1$};
\draw (c)--(b);

\node (d) at (10,0.5) {$\alpha\lp(\overline{\mathcal{T}}_2,\bar{p}^{(3)}) \rp$};
\node[circle,draw,radius=1cm] (e) at (10,0) {};
\node[circle,draw,inner sep=1pt] (f) at (9.5,-1) {$2$};
\node[circle,draw,inner sep=1pt] (g) at (10.5,-1) {$1$};
\draw (f)--(e);
\draw (g)--(e);

\node (h) at (13,1) {$\alpha\lp(\overline{\mathcal{T}}_3,\bar{p}^{(3)}) \rp$};
\node[circle,draw,radius=1cm] (i) at (13,0.5) {};
\node[circle,draw,inner sep=1pt] (j) at (13,-.5) {$2$};
\node[circle,draw,inner sep=1pt] (k) at (13,-1.5) {$1$};
\draw (k)--(j);
\draw (j)--(i);
\end{tikzpicture}
\caption[caption]{Step 3, applying $\alpha$ on the components.}
\label{fig:induction}
\end{center}
\end{figure}
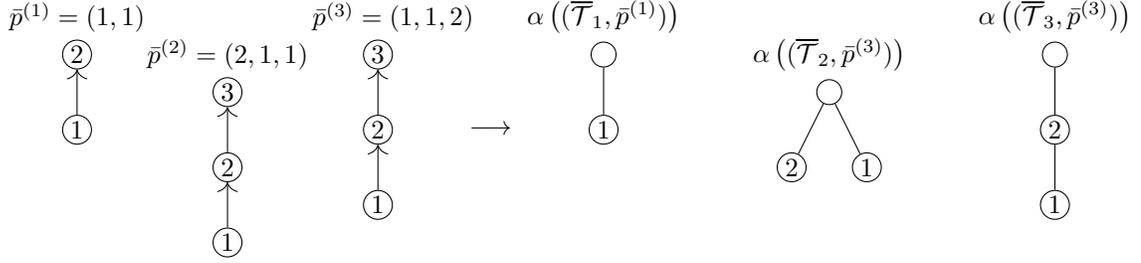

\textbf{Step 4.} For each ordered tree $\alpha((\overline{\mathcal{T}}_i,\overline{p}^(i))),$ label the root with the marked element of $A_i$ and relabel the rest of the vertices with the unmarked elements of $A_i$, preserving relative ordering. Denote these trees by $\{P_i\}_{i=1}^r$. Attach their roots to an unlabeled vertex and arrange the subtrees so that the subtree $P_i$ is to the left of $P_j$ if $i < j$. This is $\alpha((\mathcal{T},p))$.

Figure \ref{fig:relabel} shows the relabeling and Figure \ref{fig:fwdfinal} shows the final result of the running example. We constructed $\alpha$ to prove the following lemma.


\begin{figure}[h]
\begin{center}
\begin{tikzpicture}
\node (u) at (7,-1) {$A_1=\{C_3,\overline{C_6}\}$};
\node (a) at (7,1) {$\alpha\lp(\overline{\mathcal{T}}_1,\overline{p}^{(1)}) \rp$};
\node[circle,draw,radius=1cm] (b) at (7,0.5) {};
\node[circle,draw,inner sep=1pt] (c) at (7,-.5) {$1$};
\draw (c)--(b);

\node (w) at (9.8,-1.5) {$A_2 = \{\overline{C_2},C_4,C_5\}$};
\node (d) at (9.8,0.5) {$\alpha\lp(\overline{\mathcal{T}}_2,\overline{p}^{(3)}) \rp$};
\node[circle,draw,radius=1cm] (e) at (9.8,0) {};
\node[circle,draw,inner sep=1pt] (f) at (9.3,-1) {$2$};
\node[circle,draw,inner sep=1pt] (g) at (10.3,-1) {$1$};
\draw (f)--(e);
\draw (g)--(e);

\node (x) at (12.6,-2) {$A_3 = \{C_1,C_7,\overline{C_8}\}$};
\node (h) at (12.6,1) {$\alpha\lp(\overline{\mathcal{T}}_3,\overline{p}^{(3)}) \rp$};
\node[circle,draw,radius=1cm] (i) at (12.6,0.5) {};
\node[circle,draw,inner sep=1pt] (j) at (12.6,-.5) {$2$};
\node[circle,draw,inner sep=1pt] (k) at (12.6,-1.5) {$1$};
\draw (k)--(j);
\draw (j)--(i);

\node (v) at (14,-.5) {$\longrightarrow$};

\node (j) at (16,1) {$P_1$};
\node[circle,draw,inner sep=1pt] (k) at (16,0.5) {$6$};
\node[circle,draw,inner sep=1pt] (l) at (16,-.5) {$3$};
\draw (l)--(k);

\node (m) at (18.5,0.5) {$P_2$};
\node[circle,draw,inner sep=1pt] (n) at (18.5,0) {$2$};
\node[circle,draw,inner sep=1pt] (o) at (18,-1) {$5$};
\node[circle,draw,inner sep=1pt] (p) at (19,-1) {$4$};
\draw (o)--(n);
\draw (p)--(n);

\node (q) at (21,1) {$P_3$};
\node[circle,draw,inner sep=1pt] (r) at (21,0.5) {$8$};
\node[circle,draw,inner sep=1pt] (s) at (21,-.5) {$7$};
\node[circle,draw,inner sep=1pt] (t) at (21,-1.5) {$1$};
\draw (t)--(s);
\draw (s)--(r);

\end{tikzpicture}
\caption[caption]{Step 4, relabeling the trees from Figure \ref{fig:induction} with the sets from Figure \ref{fig:labelchildren}.}
\label{fig:relabel}
\end{center}
\end{figure}
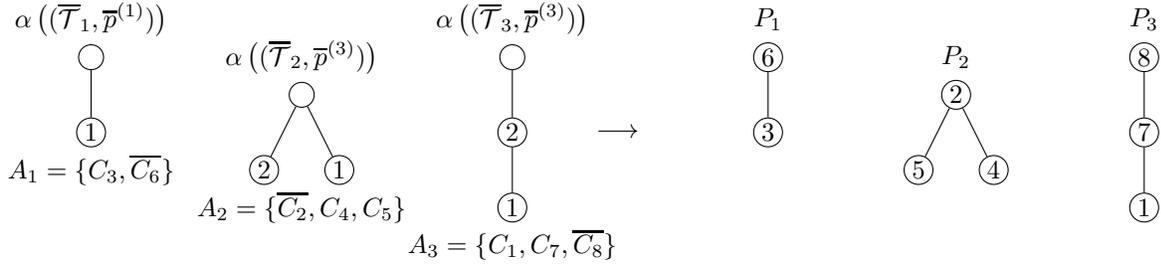

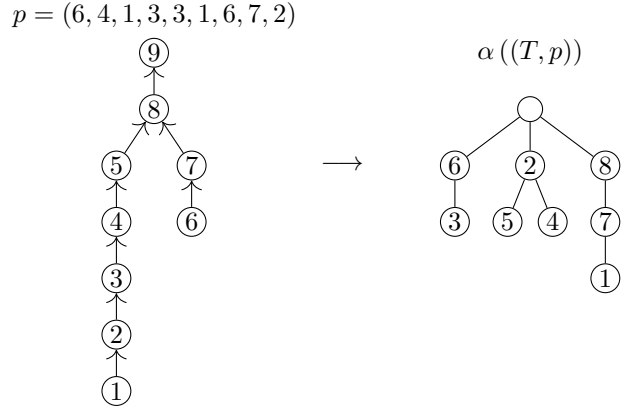
\begin{figure}[h]
\begin{center}
\begin{tikzpicture}
\node (a) at (0,0.5) {$p=(6,4,1,3,3,1,6,7,2)$};
\node[circle,draw,inner sep=1pt] (b) at (0,0) {$9$};
\node[circle,draw,inner sep=1pt] (c) at (0,-.75) {$8$};
\node[circle,draw,inner sep=1pt] (d) at (0.5,-1.5) {$7$};
\node[circle,draw,inner sep=1pt] (e) at (0.5,-2.25) {$6$};
\node[circle,draw,inner sep=1pt] (f) at (-0.5,-1.5) {$5$};
\node[circle,draw,inner sep=1pt] (g) at (-0.5,-2.25) {$4$};
\node[circle,draw,inner sep=1pt] (h) at (-0.5,-3) {$3$};
\node[circle,draw,inner sep=1pt] (i) at (-0.5,-3.75) {$2$};
\node[circle,draw,inner sep=1pt] (j) at (-0.5,-4.5) {$1$};

\draw[arrows={->[scale=1.5]}] (c)--(b);
\draw[arrows={->[scale=1.5]}] (d)--(c);
\draw[arrows={->[scale=1.5]}] (e)--(d);
\draw[arrows={->[scale=1.5]}] (f)--(c);
\draw[arrows={->[scale=1.5]}] (g)--(f);
\draw[arrows={->[scale=1.5]}] (h)--(g);
\draw[arrows={->[scale=1.5]}] (i)--(h);
\draw[arrows={->[scale=1.5]}] (j)--(i);

\node (v) at (2.5,-1.5) {$\longrightarrow$};

\node (k) at (5,0) {$\alpha\lp(T,p) \rp$};
\node[circle,draw,radius=1cm] (l) at (5,-.75) {};
\node[circle,draw,inner sep=1pt] (m) at (4,-1.5) {$6$};
\node[circle,draw,inner sep=1pt] (n) at (5,-1.5) {$2$};
\node[circle,draw,inner sep=1pt] (o) at (6,-1.5) {$8$};
\node[circle,draw,inner sep=1pt] (p) at (4,-2.25) {$3$};
\node[circle,draw,inner sep=1pt] (q) at (4.7,-2.25) {$5$};
\node[circle,draw,inner sep=1pt] (r) at (5.3,-2.25) {$4$};
\node[circle,draw,inner sep=1pt] (s) at (6,-2.25) {$7$};
\node[circle,draw,inner sep=1pt] (t) at (6,-3) {$1$};
\draw (m)--(l);
\draw (n)--(l);
\draw (o)--(l);
\draw (p)--(m);
\draw (q)--(n);
\draw (r)--(n);
\draw (s)--(o);
\draw (t)--(s);

\end{tikzpicture}
\caption[caption]{The result of $\alpha$.}
\label{fig:fwdfinal}
\end{center}
\end{figure}


\begin{lemma} \label{lem:alpha} For $n \geq 1$
\[
|\mathcal{SPF}_n| = C_{n-1}(n-1)!
\]
\end{lemma}

\begin{proof}
There are $C_{n-1}(n-1)!$ ordered trees with $n$ vertices with non-root vertices labeled by $[n-1]$. Therefore, it is sufficient to show that $\alpha$ is a bijection, which we do inductively. In the case that $|T|=1$, the only pair is $(T,1)$ and $\alpha\lp(T,1)\rp$ is an unlabeled singleton.

Now suppose $\alpha$ is a bijection for all parking functions in $\mathcal{SRP}_k$ with $k < n$ and let $(\mathcal{T},p) \in \mathcal{SRP}_n$. Run the parking procedure on $T$ for all but the final driver. Since $p$ is a prime parking function, all edges except for some on the path $\mathcal{P}$ from $p_n$ to the root have been used. Deleting these unused edges creates a collection of trees which are naturally ordered by the order their vertices appear on $\mathcal{P}$. Since we know the root component is a singleton by Proposition \ref{prop:finalspot}, we may ignore it and label the other components $\{\mathcal{T}_i\}_{i=1}^r$ for some $r$. 

We claim that $\mathcal{T}_i$ has smallest vertex label $1+\sum_{j=1}^{i-1}|\mathcal{T}_j|$ and largest label $|\mathcal{T}_i|+\sum_{j=1}^{i-1}|\mathcal{T}_j|$. Further, if $\sum_{j=1}^{i-1}|\mathcal{T}_j|$ is subtracted from each vertex, the resulting tree will be labeled by post-order. This means if we consider the subsequence of $p$, denoted $p^{(i)}$, consisting of drivers preferring $\mathcal{T}_i$, the pair $(\mathcal{T}_i,p^{(i)})$ is a relabeling of a parking function in $\mathcal{SRP}_{|\mathcal{T}_{i}|}$.

\begin{proof}[Proof of claim]
Because the edge leaving the root of every $\mathcal{T}_i$ is not used until the very last driver and $(\mathcal{T},p) \in \mathcal{SRP}_n$, the roots of the $\mathcal{T}_i$'s, denoted $\{\rho_i\}_{i=1}^r$, must lie on the left border of the tree. 

By post-order labeling, a vertex is not labeled before all vertices below and all of its left siblings are given a label. Since all vertices below $\rho_1$ belong to $\mathcal{T}_1$, all other vertices of $\mathcal{T}_1$ are labeled before $\rho_1$. No vertex in $\mathcal{T}_2$ is labeled before $\rho_1$ since $\rho_1$ is on the left border of the tree, to the left of any of its siblings. Hence, $\mathcal{T}_1$ is labeled first and is in post-order.

In general, as $\rho_{i-1}$ is on the left border of $\mathcal{T}$, the vertices below $\rho_{i-1}$ are labeled before any vertex in $\mathcal{T}_i$. Thus, $\mathcal{T}_i$ has vertices labeled $1+\sum_{j=1}^{i-1}|\mathcal{T}_j|$ to $|\mathcal{T}_i|+\sum_{j=1}^{i-1}|\mathcal{T}_j|$. Subtracting $\sum_{j=1}^{i-1}|\mathcal{T}_j|$ from each vertex is the same as labeling, via post-order, the maximal subtree with root $\rho_i$ and vertices below $\rho_{i-1}$ (inclusive) deleted, as this deletion removes a branch on the left side of the tree that is labeled before any other vertex.
\end{proof}

Let $A_1, \hdots , A_r$ be a partition of $[n-1]$ such that $j \in A_\ell$ if and only if $p_j$ is a vertex in $\mathcal{T}_\ell$. $A_i$ is the indices of the drivers preferring the component $\mathcal{T}_i$, meaning $|A_i| = |\mathcal{T}_i|$. For each deleted edge $(u,v)$, except for when $v$ is the root, if $v$ is the $k^\tx{th}$ smallest vertex in its component, $\mathcal{T}_j$, mark the $k^\tx{th}$ smallest element in $A_j$. For $A_1$, let $v =  p_n$. These marked elements track both the final driver's preference and how to reconstruct the tree from its components. The collection $\{A_i\}_{i=1}^r$ partitions $[n-1]$ and will become the labels on the resulting ordered tree.

Relabel $\{(\mathcal{T}_i,p^{(i)})\}_{i=1}^r$ by post-order (notice this is the same as subtracting $\sum_{j=1}^{i-1}|\mathcal{T}_j|$ from $v \in V(\mathcal{T}_i)$), denoting them $\{(\overline{\mathcal{T}}_i,\overline{p}^{(i)}\}_{i=1}^r$. Use the inductive hypothesis to obtain the trees $\{\alpha\lp (\overline{\mathcal{T}}_i,\overline{p}^{(i)}) \rp\}_{i=1}^r$. Label the root of $\alpha\lp(\overline{\mathcal{T}}_i,\overline{p}^{(i)}) \rp$ with the marked vertex of $A_i$, then relabel the remaining vertices with the unmarked elements of $A_i$, preserving relative order. This is possible because $|A_i| = |\mathcal{T}_i|=|\alpha\lp (\overline{\mathcal{T}}_i,\overline{p}^{(i)}) \rp|$. Attach the roots of these trees to an unmarked vertex and order them so that the tree using the labels in $A_i$ is $i^\tx{th}$-from-the-left. This tree is $\alpha\lp(\mathcal{T},p)\rp$.

To reverse, let $P$ be an ordered tree with non-root vertices labeled by $[n-1]$. Deleting the root yields several components, denoted left-to-right as $P_i$ for $1 \leq i \leq r$. The set $A_i$ is the set of labels of $P_i$ where the root of $P_i$ is the marked element. For each, delete the root's label and relabel using $[|P_i|-1]$, preserving relative order, and apply $\alpha^{-1}$ to get the collection $\{(\overline{\mathcal{T}}_i,\overline{p}^{(i)})\}_{i=1}^r$. For each $v \in \overline{\mathcal{T}}_i$ and $\overline{p}_j^{(i)}$, add $\sum_{j=1}^{i-1}|\overline{\mathcal{T}}_j|$ to recover the pairs $(\mathcal{T}_i,p^{(i)})$.

Then for $1 \leq i \leq r-1$, attach the root of $\mathcal{T}_i$ to the vertex labeled $k+\sum_{j=1}^{i-1}|\mathcal{T}_j|$, placing it to the left of any siblings, where the marked element in $A_{i+1}$ is the $k^\tx{th}$ by relative order. Attach the root of $\mathcal{T}_r$ to a singleton with label $n$, the root of $\mathcal{T}$. Let the sequence $\{{i_j}\}$ be the increasing sequence of the elements of $A_i$. Then let $p_{i_j}=p^{(i)}_j$ to recover $p$.
\end{proof}

Finally, we combine results to prove Theorem \ref{thm:main}.

\begin{proof}[Proof of Theorem \ref{thm:main}]
Let $(T,p)$ be a prime parking function, if $\phi((T,p)) = (\sigma,T_\sigma,p_\sigma)$, we define

\[
\psi(T,p) = (\sigma,\alpha((T_\sigma,p_\sigma)).
\]
The function $\psi$ is a bijection by Proposition \ref{prop:SRPn} and Lemma \ref{lem:alpha}, so we conclude that

\[
P_n = n!|\mathcal{SRP}_n| = n!\cdot (n-1)! \cdot C_{n-1} = (2n-2)!.
\]
\end{proof}

\section{Preimages of Paths Under $\alpha$} \label{sec:extras}

We investigate special families of trees and parking functions of interest.

\subsection{Preimage Of Paths}

We study what kinds of parking functions $(T,p)$  appear under $\alpha^{-1}$  when we restrict the domain to trees which are paths. For a tree of size $n+1$, there are $n!$ such paths. Recall that $\mathcal{P}_n$ is the path of $n$ vertices upon which classical parking functions are defined.

\begin{prop}\label{prop:growth}
Let $(\mathcal{P}_{n+1},s)$ be a parking function satisfying $s_1 = 1$ and $s_i \leq i-1$ for $i \geq 2$. Then $\alpha(\mathcal{P}_{n+1},s)$ is one of the $n!$ paths with non-root vertices labeled by $[n]$.
\end{prop}

\begin{proof}
That $s$ is prime is easily checked. Since $s_1 =1$, we may delete $s_1$ and consider $s' = (s_2,s_3,\hdots,s_{n+1})$. Since $s'_i = s_{i+1} \leq i$, we have $|\{i:s'_i \leq k\}| \geq k$ for any $k \in [n]$, so $s'$ is a parking function on $\mathcal{P}_n.$ Further, $s'_i$ may be one of $i$-many choices, so there are $n!$-many $s'$, and thus $s$.

Park the drivers in order. The first driver takes spot 1 and since $s_2 = 1$, the second driver takes spot 2, crossing the edge from 1 to 2. Next, since $s_2 \leq 2$, and spots 1 and 2 are taken, the third driver takes spot 3, crossing the edge from 2 to 3. Continuing this, the $i^{\tx{th}}$ driver always parks at, but never prefers, spot $i$. When the final driver parks, all spots except for $n+1$ are filled and every edge has been used except the one between $n$ and $n+1$. Thus, by the construction of $\alpha$, the root of $\alpha(s,\mathcal{P}_{n+1})$ has one child. The inductive step in the proof of Lemma \ref{lem:alpha} tells us that the shape of the tree obtained by deleting the root from $\alpha(s,\mathcal{P}_{n+1})$ is the same as that of $\alpha\lp(1,s_1,\hdots,s_{n-1}) ,\mathcal{P}_{n}\rp$. But $(1,s_1,\hdots,s_{n-1})$ is a prime parking function on $\mathcal{P}_{n}$ with the same growth property as $s$. Therefore, its root also has one child. Iterating this argument, we see that the image of such parking functions under $\alpha$ consists of paths with non-root vertices labeled by $[n]$. Since $\alpha$ is a bijection and there are $n!$ choices for $s$, all $n!$-such paths must appear.
\end{proof}

It will prove useful to have a characterization of the $s_i$ in terms of the labels of $\alpha(\mathcal{P}_{n+1},s)$.

\begin{lemma}\label{lem:alternate}
Let $(\mathcal{P}_{n+1},s)$ be a prime parking function such that $s_1 = 1$ and $s_i \leq i-1$ for $2 \leq i \leq n+1$. Further, let $\sigma \in \mathfrak{S}_n$ be the permutation given by reading the labels traveling away from the root in $\alpha(\mathcal{P}_{n+1},s)$. Denote this path $P_\sigma$. Then for $2 \leq i \leq n+1$, we have 
\[
s_{i} = |\{j > n+2-i: \sigma_j < \sigma_{n+2-i}\}|+1.
\]
\end{lemma}

\begin{proof}
 By the construction of $\alpha$ in Lemma \ref{lem:alpha}, we know $\sigma_1 = s_{n+1}$ and so $\sigma_2$ is the element in $[n]\setminus \{\sigma_1\}$ larger than exactly $s_{n}-1$ others (the $s_n^\tx{th}$ smallest element). In general, $\sigma_i$ is the $s_{n+2-i}^\tx{th}$ smallest element in $[n]\setminus \{\sigma_1,\sigma_2,\hdots,\sigma_{i-1}\}$. Thus, for $i \geq 2$, $s_i$ is given by the relative size of $\sigma_{n+2-i}$ in the set $[n]\setminus \{\sigma_1,\sigma_2,\hdots,\sigma_{n+1-i}\}$. We may write this as $s_i = |\{j > n+2-i : \sigma_j < \sigma_{n+2-i}\}|+1$.
\end{proof}

Of particular interest are the increasing parking functions that obey this growth restriction. We now examine which labellings $\sigma$ appear for $\alpha((\mathcal{P}_{n+1},s))$ with $s$ an increasing prime parking function.

\subsection{Image of Classical Increasing Parking Functions}

Borie \cite{BORIE} gives a bijection between $\mathfrak{S}_n(132),$ the permutations of length $n$ which avoid a 132 pattern, and classical increasing parking functions of length $n$. Let $\sigma \in \mathfrak{S}_n(132)$. Define for $m \in [n]$:

\[
mmp(0,m,0,0)(\sigma) = |\{i : |\{j: j<i, \sigma_j > \sigma_i\}|\geq m \}|,
\]
and set

\[
\phi(\sigma) = (mmp(0,n,0,0)+1,mmp(0,n-1,0,0)+1,\hdots, mmp(0,1,0,0)+1).
\]
Then $\phi(\sigma)$ is a classical increasing parking function and we have the following theorem due to Borie:

\begin{theorem}[Theorem 3.3, \cite{BORIE}]
$\phi$ is a bijection between $\mathfrak{S}_n(132)$ and increasing parking functions of length $n$.
\end{theorem}
We will show

\begin{theorem}
For $\sigma \in \mathfrak{S}_n(132)$, $\phi(\sigma)$ is the parking function obtained after deleting the leading 1 from  $\alpha^{-1}(P_\sigma)$.
\end{theorem}

\begin{proof}
Fix $\sigma \in \mathfrak{S}_n(132)$ and let $(\mathcal{P}_{n+1},p') = \alpha^{-1}(P_\sigma)$. By Proposition \ref{prop:growth}, we know we may write $p' = (1,p_1,p_2,\hdots,p_n)$, where $p=(p_1,p_2,\hdots,p_n)$ is a classical increasing parking function. For brevity, define for $m \in [n]$:

\begin{align*}
A_m &= \{j: j > m \tx{ and }\sigma_j < \sigma_m\},
\end{align*}
and

\begin{align*}
B_m &= \{i : |\{j: j<i, \sigma_j > \sigma_i\}|\geq m \} \\
&= \{i>m : \sigma_i \tx{ smaller than at least $m$ of } \{\sigma_1,\hdots , \sigma_{i-1}\}\}.
\end{align*}

Since $|B_m| = mmp(0,m,0,0)(\sigma)$, we have $\phi(\sigma) = (|B_n|+1,|B_{n-1}|+1,\hdots,|B_1|+1)$. From Lemma \ref{lem:alternate} we have $p_i = p_{i+1}' = |A_{n+1-i}|+1$, so it is sufficient to show that $|A_m| = |B_m|$ for $m \in [n]$. We show the sets are the same.

Fix $m \in [n]$ and let $k \in A_m$. By definition, $m < k$ and $\sigma_k < \sigma_m$. For $i < m$, if $\sigma_i < \sigma_k$, then $\sigma$ has the 132 pattern $\sigma_i\sigma_m\sigma_k$, which is not possible. Hence, $\sigma_k < \sigma_i$ for $1 \leq i \leq m$, so $k \in B_m$.

On the other hand, let $j \in B_m$. If $\sigma_j < \sigma_m$, then $j \in A_m$, so suppose $\sigma_j > \sigma_m$. Since $\sigma \in \mathfrak{S}_n(132)$, $\sigma_i < \sigma_j$ for $m+1 \leq i \leq j-1$. Thus, for indices smaller than $j$, only elements from $\{\sigma_1,\sigma_2,\hdots,\sigma_{m-1}\}$ may be larger than $\sigma_j$. However, $|\{\sigma_1,\sigma_2,\hdots,\sigma_{m-1}\}| = m-1 < m$, contradicting that $j \in B_m$. Therefore $\sigma_j < \sigma_m$, so $j \in A$.
\end{proof}

%

\section{Parking Distributions}\label{sec:PD}

Recall that a parking distribution is a parking function $(T,s)$ such that the entries of $s$ are weakly increasing. Following the decomposition in Section \ref{sec:gfproof}, we let $\widetilde{F}_n$ be the total number of parking distributions on trees with $n$ vertices, and $\widetilde{P}_n$ be the corresponding number of prime parking distributions. Set 
\[
\widetilde{F}(x) = \sum\limits_{n \geq 1} \widetilde{F}_n \dfrac{x^n}{n!} \tx{\hspace{1cm}and\hspace{1cm}} \widetilde{P}(x) = \sum\limits_{n \geq 1} \widetilde{P}_n \dfrac{x^n}{n!}.
\]
Notice that these are exponential generating functions, unlike $F(x)$ and $P(x)$, as we no longer need to account for the order of the preference sequences. Decomposing a parking distribution $(T,s)$ into the ``core'' prime component and collection of $r$ other components, as in Section \ref{sec:gfproof}, we get

\[
\widetilde{F}_n = \sum\limits_{r\geq 0}\frac{1}{r!}\sum\limits_{\sum\limits_{i=0}^rk_i=n}\widetilde{P}_{k_{0}}\widetilde{F}_{k_{1}}\cdots \widetilde{F}_{k_{r}}{n \choose k_0,k_1,\hdots,k_r} (k_0)^r.
\]
Summing over $n$, we find

\begin{equation}\label{eq:pdrelation1}
\widetilde{F}(x) = \widetilde{P}\lp x e^{\widetilde{F}(x)}\rp,
\end{equation}
which is the same relationship for parking functions in Equation \eqref{eq:composition}. We then turn our attention to prime parking distributions and prove 

\primedistr*

\begin{proof}
We let $P_n^*$ be the total number of tuples $(T,p,v)$, called \emph{marked} prime parking distributions, where $(T,p)$ is a prime parking distribution and $v$ is a leaf of $T$ with $|T|=n$. Define the exponential generating function
\[
P^*(x) = \sum\limits_{n\geq 1}P_n^*\frac{x^n}{n!}.
\]

We count $P_n^*$ in two ways in order to determine the coefficients of $\widetilde{P}(x)$. We may construct a marked prime parking distribution from a prime parking distribution $(T,p)$ with $|T|=n-1$ by ``growing" the marked leaf from some vertex in $T$. We choose one vertex $w$ in $T$ which has $j \geq 1$ drivers preferring it, select a label for the marked leaf $v$, attach $v$ as a child, add a driver preferring $v$, and change $1 \leq i \leq j$ drivers preferring $w$ to prefer $v$ instead. The number of choices we can make for $v$'s parents and number of drivers whose preferences we change is $\sum\limits_{w \in V(T)}|\{ i : p_i = w\}| = n-1$. On the other hand, any marked prime parking distribution can be changed to a regular prime parking distribution by deleting the marked vertex $v$, deleting one driver preferring it, and changing the preference of other drivers preferring $v$ to instead prefer $v$'s parent. Notice since the marked parking distribution is prime, at least two drivers prefer the marked leaf. Figure \ref{fig:marked} has several examples where the shaded vertex is added and marked letters of $p$ are drivers whose preference was changed.
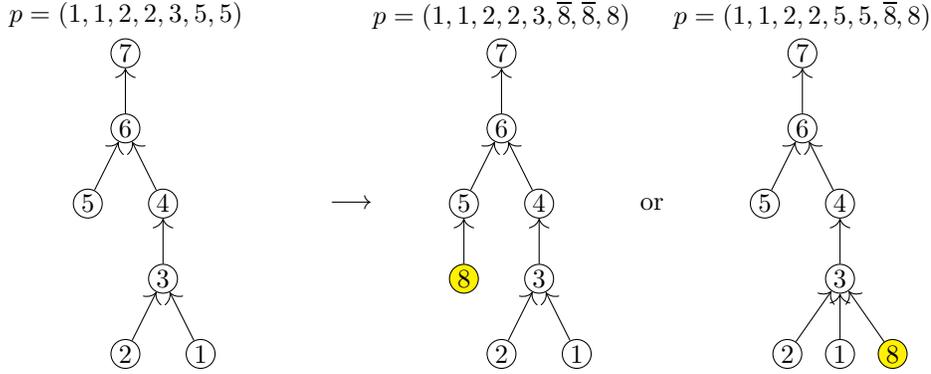
\begin{figure}[h!]
\begin{center}
\begin{tikzpicture}

\node[circle,draw,inner sep=1pt] (a) at (0,0) {7};
\node[circle,draw,inner sep=1pt] (b) at (0,-1) {6};
\node[circle,draw,inner sep=1pt] (c) at (-.5,-2) {5};
\node[circle,draw,inner sep=1pt] (d) at (.5,-2) {4};
\node[circle,draw,inner sep=1pt] (e) at (.5,-3) {3};
\node[circle,draw,inner sep=1pt] (f) at (0,-4) {2};
\node[circle,draw,inner sep=1pt] (g) at (1,-4) {1};
\node (h) at (0,.5) {$p=(1,1,2,2,3,5,5)$};

\draw[arrows={->[scale=1.5]}] (b)--(a);
\draw[arrows={->[scale=1.5]}] (c)--(b);
\draw[arrows={->[scale=1.5]}] (d)--(b);
\draw[arrows={->[scale=1.5]}] (e)--(d);
\draw[arrows={->[scale=1.5]}] (f)--(e);
\draw[arrows={->[scale=1.5]}] (g)--(e);

\node[circle,draw,inner sep=1pt] (i) at (5,0) {7};
\node[circle,draw,inner sep=1pt] (j) at (5,-1) {6};
\node[circle,draw,inner sep=1pt] (k) at (4.5,-2) {5};
\node[circle,draw,inner sep=1pt] (l) at (5.5,-2) {4};
\node[circle,draw,inner sep=1pt] (m) at (5.5,-3) {3};
\node[circle,draw,inner sep=1pt] (n) at (5,-4) {2};
\node[circle,draw,inner sep=1pt] (o) at (6,-4) {1};
\node[circle,draw,fill=yellow,inner sep=1pt] (p) at (4.5,-3) {8};
\node (q) at (5,.5) {$p=(1,1,2,2,3,\overline{8},\overline{8},8)$};

\draw[arrows={->[scale=1.5]}] (j)--(i);
\draw[arrows={->[scale=1.5]}] (k)--(j);
\draw[arrows={->[scale=1.5]}] (l)--(j);
\draw[arrows={->[scale=1.5]}] (m)--(l);
\draw[arrows={->[scale=1.5]}] (n)--(m);
\draw[arrows={->[scale=1.5]}] (o)--(m);
\draw[arrows={->[scale=1.5]}] (p)--(k);

\node (r) at (3,-2) {$\longrightarrow$};
\node (s) at (7,-2) {or};

\node[circle,draw,inner sep=1pt] (t) at (9,0) {7};
\node[circle,draw,inner sep=1pt] (u) at (9,-1) {6};
\node[circle,draw,inner sep=1pt] (v) at (8.5,-2) {5};
\node[circle,draw,inner sep=1pt] (w) at (9.5,-2) {4};
\node[circle,draw,inner sep=1pt] (x) at (9.5,-3) {3};
\node[circle,draw,inner sep=1pt] (y) at (8.8,-4) {2};
\node[circle,draw,inner sep=1pt] (z) at (9.5,-4) {1};
\node[circle,draw,fill=yellow,inner sep=1pt] (aa) at (10.2,-4) {8};
\node (ab) at (9,.5) {$p=(1,1,2,2,5,5,\overline{8},8)$};

\draw[arrows={->[scale=1.5]}] (u)--(t);
\draw[arrows={->[scale=1.5]}] (v)--(u);
\draw[arrows={->[scale=1.5]}] (w)--(u);
\draw[arrows={->[scale=1.5]}] (x)--(w);
\draw[arrows={->[scale=1.5]}] (y)--(x);
\draw[arrows={->[scale=1.5]}] (z)--(x);
\draw[arrows={->[scale=1.5]}] (aa)--(x);

\end{tikzpicture}
\caption{Two possibilities when adding the marked leaf.}
\label{fig:marked}
\end{center}
\end{figure}

This means for $n \geq 2$,

\[
P_n^* = n(n-1)\widetilde{P}_{n-1},
\]
so noting $P_0^* = 0$ and $P_1^*=1$ and summing over $n$, we get

\begin{equation}\label{eq:marked1}
P^*(x) = x +x^2\widetilde{P}'(x).
\end{equation}

For the second equation, we decompose a marked prime parking distribution $(T,p,v)$ as in Section \ref{subsec:alpha}: park all drivers except for one preferring $v$. For a general picture, see Figure \ref{fig:distr_decomp}. Dashed edges denote those unused before the final driver. As before, the edges which have not yet been used define $r+1$, for some $r \geq 0$, components of size $k_i$ with prime parking functions, one of which is marked. Accounting for the edges connecting the components, the label of the root, and the labels on the components, we have for $n \geq 2$,
\[
P_n^* = n \sum\limits_{r \geq 0}\sum\limits_{\sum k_i = n-1} {n-1 \choose k_0,k_1,\hdots,k_r} P_{k_{0}}^*\widetilde{P}_{k_{1}}\cdots \widetilde{P}_{k_{r}}k_1k_2\cdots k_r,
\]
so summing over $n$ with $P_0^* = 0$ and $P_1^* = 1$, we get

\begin{equation}\label{eq:marked2}
P^*(x) = x + \frac{xP^*(x)}{1-x\widetilde{P}'(x)}.
\end{equation}

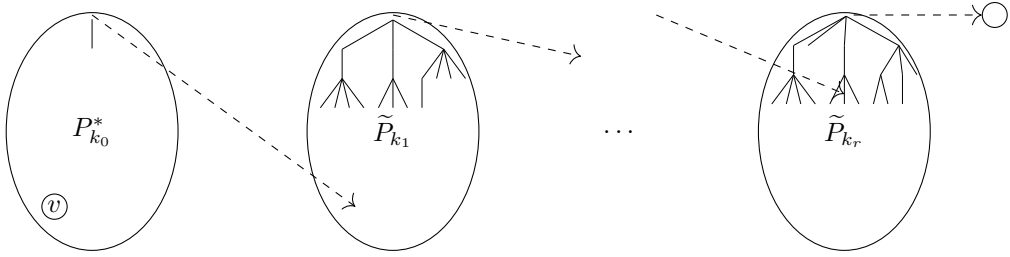
\begin{figure}[h]
\begin{center}
\begin{tikzpicture}
\node[ellipse,draw, minimum width=65pt, minimum height=90pt] at (0,0) {$P_{k_{0}}^*$};
\node[ellipse,draw, minimum width=65pt, minimum height=90pt] at (4,0) {$\widetilde{P}_{k_{1}}$};
\node[ellipse,draw, minimum width=65pt, minimum height=90pt] at (10,0) {$\widetilde{P}_{k_{r}}$};
\node[circle,draw,inner sep=1pt] at (-.5,-1) {$v$};

\node (a) at (10,1.55) {};
\node[circle,draw,radius = 1cm] (b) at (12,1.55) {}; 
\node at (7,0) {$\hdots$};

\draw[arrows={->[scale=1.5]},dashed] (a) -- (b);
\draw[arrows={->[scale=1.5]},dashed] (0,1.55) -- (3.5,-1);
\draw[arrows={->[scale=1.5]},dashed] (4,1.55) -- (6.5,1);
\draw[arrows={->[scale=1.5]},dashed] (7.5,1.55) -- (10,.5);

\node at (0,1.3){
\scalebox{.4}{
\begin{forest}
random tree/.style n args={2}{
if={#1>0}{repeat={random(0,#2)}{append={[,random tree={#1-1}{#2}]}}}{},
  parent anchor=center, child anchor=center, grow=south},
[,random tree={3}{4}]
\end{forest}
}
};

\node at (4,0.9){
\scalebox{.4}{
\begin{forest}
random tree/.style n args={2}{
if={#1>0}{repeat={random(0,#2)}{append={[,random tree={#1-1}{#2}]}}}{},
  parent anchor=center, child anchor=center, grow=south},
[,random tree={3}{4}]
\end{forest}
}
};

\node at (10,0.95){
\scalebox{.4}{
\begin{forest}
random tree/.style n args={2}{
if={#1>0}{repeat={random(0,#2)}{append={[,random tree={#1-1}{#2}]}}}{},
  parent anchor=center, child anchor=center, grow=south},
[,random tree={3}{4}]
\end{forest}
}
};

\end{tikzpicture}
\caption{Decomposition based on final driver's movement.}
\label{fig:distr_decomp}
\end{center}
\end{figure}
Combining Equations \eqref{eq:marked1} and \eqref{eq:marked2}, we see

\[
x\lp\widetilde{P}'(x)\rp^2 + (x-1)\widetilde{P}'(x) +1 = 0,
\]
and so as $\widetilde{P}'(0) = 1$,

\[
\widetilde{P}'(x) = \frac{1-x-\sqrt{x^2-6x+1}}{2x},
\]
which is the ordinary generating function for the large Schr{\"o}der numbers. Hence,
\[
\widetilde{P}'(x) = \sum\limits_{n \geq 0} n!S_n\frac{x^n}{n!},
\]
meaning
\[
\widetilde{P}(x) = \sum\limits_{n \geq 1} (n-1)!S_{n-1}\frac{x^n}{n!}.
\]
\end{proof}

In addition to the relationship in Equation \eqref{eq:pdrelation1}, we can let $F_n^*$ denote the total number of parking distributions on trees with $n$ vertices with one leaf marked and let
\[
F^*(x) =\sum\limits_{n \geq 1}F_n^*\frac{x^n}{n!}.
\]
Constructing a parking distribution counted by $F_n^*$ by ``growing" it from a parking distribution counted by $\widetilde{F}_{n-1}$ gives $n$ choices for the leaf label, $n-1$ choices of nodes to attach the marked leaf to without reassigning drivers, and $n-1$ choices for attaching the marked leaf to a node and reassigning at least one driver from the parent node. Thus, $F_n^* = 2n(n-1)\widetilde{F}_{n-1},$ so
\begin{equation}\label{eq:Fstar1}
F^*(x) = x + 2x^2\widetilde{F}'(x).
\end{equation}

On the other had, if a parking distribution on a marked tree with $n$ vertices is decomposed by considering the final unfilled node when the drivers on the marked leaf park last, the final node is of one of three types: the root, the marked leaf, or neither. Therefore, we have for $n \geq 2$
\[
\begin{split}
F_n^* &= \sum\limits_{r\geq0}\frac{n}{r!} \sum\limits_{\sum\limits_{i=0}^rk_i=n-1}{n-1 \choose k_0,\hdots,k_r}F_{k_{0}}^*\widetilde{F}_{k_{1}}\cdots \widetilde{F}_{k_{r}} \\
&+ n(n-1)F_{n-1} \\
&+ \delta_{n \geq 3}\sum\limits_{r\geq0}\frac{n}{r!} \sum\limits_{k+\sum\limits_{i=0}^rk_i=n-1}{n-1 \choose k_0,\hdots,k_r,k}F_{k_{0}}^*\widetilde{F}_{k_{1}}\cdots \widetilde{F}_{k_{r}}\widetilde{F}_kk,
\end{split}
\]
where $\delta_{n\geq3}$ is 1 if $n\geq 3$ and 0 otherwise. Summing over $n$ we get

\begin{equation}\label{eq:Fstar2}
F^* = x + xF^*e^F +x^2F' + x^2F^*F'e^F.
\end{equation}
Combining Equations \eqref{eq:Fstar1} and \eqref{eq:Fstar2}, we obtain a differential equation satisfied by $\widetilde{F}$,

\[
\widetilde{F}' = e^{\widetilde{F}} (1+x\widetilde{F}')(1+2x\widetilde{F}'),
\]
with initial condition $\widetilde{F}(0) = 0$. We note that the generating function for general tree parking functions, $F(x)$, satisfies the similar equation

\[
F' = e^F(1+xF')^2.
\]
For details, see Equation (7) in \cite{LACKNER2016}.

%
%
%

\section{Conclusion} \label{sec:conclusion}

This paper extends the results by Lackner and Panholzer \cite{LACKNER2016} to prime parking functions on trees, from which all tree parking functions can be constructed. Additionally, we enumerate prime parking distributions, the natural generalization of increasing classical parking functions. There are many open questions regarding tree parking functions, of which we mention several:

\begin{enumerate}



\item Find an ``interesting" direct bijection from prime tree parking functions to $\mathfrak{S}_{2n-2}$, preferably preserving some statistics.

\item Are there families of trees, such as caterpillars or binary trees, which support an interesting sequence of prime parking functions?

\item In \cite{LACKNER2016}, the authors find bounds for the number of parking functions, $S(T,n)$, on a given tree $T$ and ask if $S(T,n)$ has a ``simple characterization". We ask the same question for prime parking functions on a given tree. 

\item Butler, Graham, and Yan \cite{BUTLER2017} consider parking distributions on trees in the case where $n+k$ drivers attempt to park and exactly $k$ do not find a space. Can their results be refined to consider only prime parking distributions?
\end{enumerate}


\bibliographystyle{abbrv}
\bibliography{PrimeReferences}

 
\end{document}